\RequirePackage{fix-cm}
\documentclass[smallextended,numbook]{svjour3}          % onecolumn (second format)
\smartqed  % flush right qed marks, e.g. at end of proof
\usepackage{graphicx}
\usepackage{latexsym}
\usepackage{amssymb}
\usepackage{amsmath}
\usepackage{relsize}
\usepackage{cite}

\newcommand{\nc}{\newcommand}
\nc{\nt}{\newtheorem}
\nt{thm}{Theorem}[section]
\nt{cor}[thm]{Corollary}
\nt{prop}[thm]{Proposition}
\nt{obs}[thm]{Observation}
\nt{lem}[thm]{Lemma}
\nt{defn}[thm]{Definition}
\nt{exa}[thm]{Example}
\nt{rem}[thm]{Remark}
\nt{ass}[thm]{Assumption}
\nt{alg}[thm]{Algorithm}
\nt{con}[thm]{Conjecture}
\nc{\ip}[2]{\mbox{$\langle #1,#2 \rangle$}}
\nc{\pf}{\noindent{\bf Proof\ \ }}
\nc{\finpf}{\hfill{$\Box$}\linespace}
\nc{\linespace}{\vspace{\baselineskip} \noindent}
\nc{\R}{{\bf R}}
\nc{\tto}{\rightrightarrows}
\nc{\cl}{\mbox{\rm cl}\,}
\nc{\cls}{ \mbox{{\scriptsize {\rm cl}}}\,}
\nc{\conv}{\mbox{\rm conv}}
\nc{\rb}{\mbox{\rm rb}\,}
\nc{\ri}{\mbox{\rm ri}\,}
\nc{\inter}{\mbox{\rm int}\,}
\nc{\kernel}{\mbox{\rm ker}\,}
\nc{\bd}{\mbox{\rm bd}\,}
\nc{\spann}{\mbox{\rm span}\,}
\nc{\rint}{\mbox{\rm rint}\,}
\nc{\epi}{\mbox{\rm epi}\,}
\nc{\gph}{\mbox{\rm gph}\,}
\nc{\rge}{\mbox{\rm rge}\,}
\nc{\rgel}{\mbox{\rm {\scriptsize rge}}\,}
\nc{\sepi}{\mbox{\rm {\scriptsize epi}}\,}
\nc{\sbd}{\mbox{\rm {\scriptsize bd}}\,}
\nc{\dom}{\mbox{\rm dom}\,}
\nc{\lin}{\mbox{\rm lin}\,}
\nc{\detr}{\mbox{\rm det}\,}
\nc{\para}{\mbox{\rm par}\,}
\nc{\crit}{\mbox{\rm crit}\,}
\nc{\cone}{\mbox{\rm cone}\,}
\nc{\diag}{\mbox{\rm Diag}\,}
\nc{\fix}{\mbox{\rm Fix}}
\nc{\rank}{\mbox{\rm rank}\,}
\nc{\argmin}{\mbox{\rm argmin}\,}

\def\ox{\bar{x}}

\def\oy{\bar{y}}

\def\sd{\partial}
\def\xyb{(\bar x,\bar y)}

\def\kk{\kappa}

\def\ran{\rangle}
\def\lan{\langle}

\def\Pi{{\cal P}}
\def\B{{\cal B}}

\def\X{{\mathcal X}}

\def\la{\lambda}
\def\al{\alpha}

\def\ep{\varepsilon}
\def\sig
\def\sd{\partial}
\def\xb{\bar{x}}

\def\yb{\bar{y}}

\def\dom{{\rm dom}~}

\def\rra{\rightrightarrows}

\def\ep{\varepsilon}
\def\la{\lambda}

\def\al{\alpha}

\def\vf{\varphi}

\def\sig{\sigma}

\def\cl{{\rm cl}}
\def\cone{{\rm cone}}
\def\dom{{\rm dom}~}
\def\epi{{\rm epi}~}

\def\bd{{\rm bd}}

\def\conv{{\rm conv}~}
\def\ri{{\rm ri}~}
\def\ran{\rangle}
\def\lan{\langle}

\def\rank{{\rm rank}~}
\def\crit{\rm Crit}

\newcommand{\lf}{\operatornamewithlimits{liminf}}
\newcommand{\ls}{\operatornamewithlimits{limsup}}

\begin{document}

\title{Quadratic growth and critical point stability of semi-algebraic functions\thanks{Research of A.D. Ioffe was supported in part by the US-Israel Binational Science Foundation Grant 2008261. Work of Dmitriy Drusvyatskiy on this paper has also been partially supported by the US-Israel Binational Science Foundation Travel Grant for Young Scientists. }
}
%\subtitle{Do you have a subtitle?\\ If so, write it here}

%\titlerunning{Short form of title}        % if too long for running head

\author{D. Drusvyatskiy         \and
        A.D. Ioffe %etc.
}

%\authorrunning{Short form of author list} % if too long for running head

\institute{D. Drusvyatskiy \at
%               Department of Combinatorics and Optimization,
%    University of Waterloo,
%    200 University Avenue West, Office MC 4012,
%    Waterloo, ON N2L 3G1, Canada;\\
    Department of Mathematics, University of Washington, Seattle, WA 98195; \\
              Tel.: (206) 543-1150\\
              Fax: (206) 543-0397\\
              \email{ddrusv@uw.edu}\\
              {\tt math.washington.edu/{\raise.17ex\hbox{$\scriptstyle\sim$}}ddrusv}.          
           \and
           A.D. Ioffe \at
              Department of Mathematics, Technion-Israel Institute of Technology, Haifa, Israel 32000;\\
Tel:972-4-8294099\\
Fax:972-4-8293388\\
\email{ioffe@tx.technion.ac.il}\\
{\tt http://www.math.technion.ac.il/Site/people/process.php?id=672}
}

\date{Received: date / Accepted: date}
% The correct dates will be entered by the editor

\maketitle

\begin{abstract}
We show that quadratic growth of a semi-algebraic function is equivalent to strong metric subregularity of the subdifferential --- a kind of stability of generalized critical points. In contrast, this equivalence can easily fail outside of the semi-algebraic setting. 
As a consequence, we derive necessary conditions and sufficient conditions for optimality in subdifferential terms. 
% \PACS{PACS code1 \and PACS code2 \and more}
\end{abstract}

\section{Introduction}
Quadratic growth conditions of extended-real-valued functions play a central role in nonlinear optimization, both for convergence analysis of algorithms and for perturbation theory. See for example \cite{Bon_Shap, rusz, NW}. Classically, a point $\bar{x}$ is a {\em strong local minimizer} of a function $f$ on $\R^n$ if there exists a constant $\alpha >0$ and a neighborhood $U$ of $\bar{x}$ such that the inequality 
\begin{equation}\label{eqn:qgrowth}
f(x)\geq f(\bar{x})+\alpha \|x-\bar{x}\|^2 \quad \textrm{ holds for all } x\in U.
\end{equation}
Here $\|\cdot\|$ denotes the Euclidean norm on $\R^n$. For $C^2$-smooth functions $f$ this condition simply amounts to requiring $\bar{x}$ to be a critical point with the Hessian $\nabla^2 f(\bar{x})$ being positive definite. %Our goal in the current work is to elucidate the role of strong local minimizers in nonsmooth optimization.

What condition would then entail second order growth of the function near a minimizer, when derivatives do not exist? In a seminal paper \cite{AG}, Arag\'on Artacho and Geoffroy gave an elegant and rather unexpected answer to this question for closed convex functions:  a minimizer $\bar{x}$ of a closed convex function $f$ is a strong local minimizer if and only if the convex subdifferential $\partial f$ is {\em strongly subregular} at $(\bar{x},0)$. The latter  means that there exists a constant $\kappa \geq 0$ and a neighborhood $V$ of $\bar{x}$ so that the inequality
\begin{equation}\label{eqn:strong_met_sub}
\|x-\bar{x}\|\leq \kappa\cdot d\big(0;\partial f(x)\big) \quad \textrm{ holds for all } x\in V,
\end{equation}
where $d\big(0;\partial f(x)\big)$ denotes the minimal norm of subgradients $v\in\partial f(x)$.
Hence $\bar{x}$ being a strong minimizer amounts to requiring existence of a linear bound on the distance of a putative solution $x$ to $\bar{x}$ in terms of the ``distance to criticality'' $d\big(0;\partial f(x)\big)$ --- an appealing property. %An entirely analogous characterization is valid for $C^2$-smooth functions. 

%In turn, condition (\ref{eqn:strong_met_sub}) holds if and only if there exist a (possibly slightly larger) constant $\kappa \geq 0$ and a neighborhood $V$ of $\bar{x}$ such that the inclusion
%\begin{equation}\label{eqn:stab}
%(\partial f)^{-1}(w)\cap V\subset \bar{x}+\kappa \|w\|\B, \quad \textrm{ holds for all small }w\in\R^n,
%\end{equation}
%where $\B$ is a closed unit ball of $\R^n$. See \cite[Section~3I]{imp} for details. 
%Thus being a strong local minimizer of a convex function is a minimizing points stability condition: it asserts control over how the minimizing points of the perturbed function $f_w(x)=f(x)-w^{T}x$ can deviate from $\bar{x}$ in terms of the deviation of the parameter $w$ away from the origin --- an appealing sensitivity property.

%It is well known that conditions (\ref{eqn:strong_met_sub}) and (\ref{eqn:stab}) are equivalent without convexity, provided that we interpret $\partial f$ as the limiting Fr\'{e}chet subdifferential \cite[Theorem~3I.2]{imp}. 

%In the current work, we aim to address the question: to what extent is such a transparent and mathematically rigorous interpretation of quadratic growth true more generally? The principal result of the paper states that a similar property holds also for semi-algebraic functions (or more generally, functions
%definable in o-minimal structures), but now stated in terms of the limiting subdifferential $\partial f$.

In the current work, we aim to address the question:
to what extent is such a transparent and mathematically rigorous interpretation of quadratic growth true more generally beyond the convex setting? In this case, naturally one replaces the convex subdifferential with the limiting subdifferential in the sense of variational analysis; see for example \cite{VA}.
A partial result was obtained in \cite{AG_infin,Dima_Ng}, where the authors showed that given a local minimizer $\bar{x}$ of a closed function $f$, if the subdifferential $\partial f$ is strongly subregular at $(\bar{x},0)$, then $\bar{x}$ is indeed a strong local minimizer.
On the other hand, the converse may easily fail, in contrast to both the convex and smooth situations; case in point,  for $f(x):=2x^2+\frac{1}{2}x^2\sin\big(\frac{1}{x}\big)$. See Example~\ref{ex:fail} for more details. Such badly behaved functions, however, rarely appear in practice. The usual way to eliminate such functions from consideration is to restrict attention to favorable function classes, such as to those that are {\em amenable}~\cite{amen} or {\em prox-regular}~\cite{prox_reg}. 

In the current work, we take a different approach. We consider prototypical nonpathological functions, namely those that are {\em semi-algebraic} --- meaning that their epigraphs are composed of finitely many sets, each defined by finitely many polynomial inequalities. See \cite{Coste-semi, BCR} for more on semi-algebraic geometry. This class subsumes neither the case of convex nor the case of $C^2$-smooth functions. Nevertheless, it has great appeal. Semi-algebraic functions are common, they are easy to recognize (via the Tarski-Seidenberg Theorem on quantifier elimination), and in sharp contrast to the usual favorable function classes of variational analysis (e.g. prox-regular), semi-algebraic functions do not need to be Clarke-regular\cite{VA}. For a discussion on the role of semi-algebraic geometry in nonsmooth optimization, see \cite{tame_opt}. Our principle result (described in Theorem~\ref{thm:mainres}\footnote{This theorem was proved in early 2012, and presented in June 2012 at the conference ``Constructive Nonsmooth Analysis and Related Topics'' in Saint Petersburg, Russia, but we had to delay further work over this subject as there were two more manuscripts (in a more advanced
state) that demanded our attention.  Shortly after that, the first author in conversations with Nghia became aware of the very recent results of Mordukhovich-Nghia leading to \cite{Dima_Ng}.}) shows that for semi-algebraic functions satisfying a minimal continuity condition, strong local minimality and strong subregularity of the subdifferential at local minimizers are equivalent properties. 

\parskip=0pt
 Our development, in turn, immediately leads to another question as there is a substantial difference between applications of the convex and nonconvex results. Indeed, in the convex case strong minimality is characterized in terms of the subdifferential alone: $x$ is a strong local minimizer if and only if $0\in\partial f(\bar{x})$ and $\partial f$ is strongly subregular at $(\bar{x},0)$.
In the nonconvex semi-algebraic case the replacement of the assumption that $\bar{x}$ is
a local minimizer by the condition $0\in \partial f(\bar{x})$ is obviously impossible (consider for instance $f(x) = x|x|$).
So the question is whether it is possible to characterize \eqref{eqn:qgrowth} in purely subdifferential
terms, that is without having to assume that the point in question is already a local minimizer. In the smooth situation, one arrives at the standard second-order optimality conditions. Following this trend of thought, one might be led to a second-order subdifferential in the formulation of analogous conditions; see \cite[Theorem 1.3]{tilt}. This however leads to conditions characterizing ``strong metric regularity'' rather than the weaker subregularity concept.
In Theorem~\ref{thm:nec_suff}, we develop necessary conditions
and sufficient condition for \eqref{eqn:qgrowth} to hold, based only on first-order subdifferential considerations. These turn out to be more precise than conditions based on second-order subdifferentials. See Section~\ref{sec:nec_suf} for more details.
%These subgradient conditions reduce in the smooth case to second order conditions for optimality (appropriately rested). 
As in the smooth case, there is a gap between the conditions which however is small enough to contain the classical conditions as a special case. 

To be concrete, we state our results for semi-algebraic functions.  Analogous results, with essentially identical proofs, hold for functions definable in an ``o-minimal structure'' and, more generally, for ``tame'' functions. See \cite{DM} for more on the subject. The outline of the manuscript is as follows. In Section~\ref{sec:prelim}, we record basic elements of variational analysis and semi-algebraic geometry that we will need throughout the manuscript. In Section~\ref{sec:main}, we prove that in the semi-algebraic setting strong metric subregularity of the subdifferential at a local minimizer is equivalent to strong local minimality. In Section~\ref{sec:imply}, we elucidate an argument of \cite[Corollary 3.2]{Dima_Ng} showing that strong subregularity of the subdifferential at a local minimizer, for virtually any subdifferential of interest, implies a quadratic growth condition.
In Section~\ref{sec:nec_suf}, we discuss necessary and sufficient conditions for optimality.

\section{Preliminaries}\label{sec:prelim}
\subsection{Some elements of Variational Analysis}
In this subsection, we summarize some of the fundamental tools used in variational analysis and nonsmooth optimization.
We refer the reader to the monographs \cite{Borwein-Zhu,CLSW, Mord_1, penot_book, VA} for more details.  Unless otherwise stated, we follow the terminology and notation of \cite{VA}.

Throughout $\R^n$ will denote the $n$-dimensional Euclidean space, with the inner-product written as $\langle\cdot,\cdot \rangle$. We will denote the induced norm by $\|\cdot\|$. The closed ball centered at $x\in \R^n$ of radius $r$ will be denoted by $\B_{r}(x)$ while the closed unit ball centered at the origin will be denoted by $\B$. 
The {\em distance } of a point $x\in\R^n$ to a set $Q\subset \R^n$ is the quantity
$$
d(x; Q):=\inf_{y\in Q} \|x-y\|,
$$
with the convention $d(x,\emptyset)=+\infty$.
A set-valued mapping $F\colon \R^n \tto \R^m$ is a mapping assigning to each point $x\in \R^n$ a subset $F(x)\subset \R^m$. The {\em domain} and {\em graph} of $F$ are defined by
\begin{align*}
\dom F &:=\big\{x\in \R^n :F(x)\ne\emptyset\big\}, \\ 
\gph F &:=\big\{(x,y)\in \R^n\times \R^m : y\in F(x)\big\},
\end{align*}
respectively. 

\begin{defn}[Strong metric subregularity]\label{mr-def} {\hfill \\}{\rm Consider a mapping $F\colon \R^n\tto \R^m$ and a pair $(\ox,\oy)\in \gph F$. We say that $F$ is {\em strongly metrically subregular} at $(\ox,\oy)$ {\em with modulus} $\kappa \geq 0$ if there is a real number $\ep > 0$ such that the inequality
\begin{equation*}\label{ms}
\|x-\bar{x}\|\le\kk d\big(\oy;F(x)\big)\quad\quad\mbox{ holds for all}\quad x\in\B_{\ep}(\bar{x}).
\end{equation*}
}
\end{defn}

Observe that in particular, strong metric subregularity of $F\colon \R^n\tto \R^m$ at $(\ox,\oy)$ implies that $\bar{x}$ is a locally isolated point of $F^{-1}(\bar{y})$. 
Given a set $Q\subset\R^n$ and a point $\bar{x}\in Q$, we will denote by ``$o(\|x-\bar{x}\|) \textrm{ for }x\in Q$'' a term with the property that 
$$\frac{o(\|x-\bar{x}\|)}{\|x-\bar{x}\|}\rightarrow 0\qquad \textrm{when } x\stackrel{Q}{\rightarrow} \bar{x} \textrm{ with }x\neq\bar{x}.$$

\begin{defn}[Normal cones]
{\rm Consider a set $Q\subset\R^n$ and a point $\bar{x}\in Q$. 
\begin{itemize}
\item The {\em Fr\'{e}chet normal cone} to $Q$ at $\bar x$, denoted 
$\hat N_Q(\bar x)$, consists of all vectors $v \in \R^n$ satisfying $$\langle v,x-\bar{x} \rangle \leq o(\|x-\bar{x}\|) \quad\textrm{ for }x\in Q.$$ 
\item The {\em limiting normal cone} to $Q$ at $\bar{x}$, denoted $N_Q(\bar{x})$, consists of all vectors $v\in\R^n$ such that there are sequences $x_i\stackrel{Q}{\rightarrow} \bar{x}$ and $v_i\rightarrow v$ with $v_i\in\hat{N}_Q(x_i)$.
%\item The {\em Clarke normal cone} to $Q$ at $\bar{x}$, denoted $N_Q^c(\bar{x})$, is defined by $$N_Q^c(\bar{x}):=\textrm{cl conv }N_Q(\bar{x}).$$
\end{itemize}
} 
\end{defn}

Given any set $Q\subset\R^n$ and a mapping $f\colon Q\to \widetilde{Q}$, where $\widetilde{Q}\subset\R^m$, we say that $f$ is $C^p$-{\em smooth} (for $p=1,\ldots, \infty$) if for each point $\bar{x}\in Q$, there is a neighborhood $U$ of $\bar{x}$ and a $C^p$-smooth mapping $\hat{f}\colon \R^n\to\R^m$ that agrees with $f$ on $Q\cap U$. If a $C^p$-smooth function $f$ is bijective and its inverse is also $C^p$-smooth, then we say that $f$ is a $C^p$-{\em diffeomorphism}. 

\begin{defn}[Manifolds]
{\rm Consider a set $M\subset\R^n$. We say that $M$ is a $C^p$ {\em manifold} (or ``embedded submanifold'') of dimension $r$ if for each point $\bar{x}\in M$, there is an open neighborhood $U$ around ${\bar{x}}$ such that $M\cap U=F^{-1}(0)$, where $F\colon U\to\R^{n-r}$ is a $C^p$-smooth mapping with $\nabla F(\bar{x})$ of full rank. 
}
\end{defn}

If $M$ is a $C^1$ manifold, then for every point $\bar{x}\in M$, the normal cones $\hat{N}_M(\bar{x})$ and $N_M(\bar{x})$ are equal to the normal space to $M$ at $\bar{x}$, in the sense of differential geometry \cite[Example 6.8]{VA}.

Consider the extended real line $\overline{\R}:=\R\cup\{-\infty\}\cup\{+\infty\}$. We say that an extended-real-valued function is proper if it is never $\{-\infty\}$ and is not always $\{+\infty\}$.  For a function $f\colon\R^n\rightarrow\overline{\R}$, we define the {\em domain} of $f$ to be $$\mbox{\rm dom}\, f:=\{x\in\R^n: f(x)<+\infty\},$$ and we define the {\em epigraph} of $f$ to be $$\mbox{\rm epi}\, f:= \{(x,r)\in\R^n\times\R: r\geq f(x)\}.$$
A function $f\colon\R^n\to\overline{\R}$ is {\em lower-semicontinuous} (or {\em lsc} for short) at $\bar{x}$ if the inequality $\lf_{x\to\bar{x}} f(x)\geq f(\bar{x})$ holds.

The following is the key notion we study in the current work.
\begin{defn}[Strong local minimizer]
{\rm 
A point $\bar{x}$ is a {\em strong local minimizer }of a function $f\colon\R^n\to\overline{\R}$ if there exists a constant $\alpha >0$ and a neighborhood $U$ of $\bar{x}$ such that the inequality 
$$f(x)\geq f(\bar{x})+\frac{\alpha}{2} \|x-\bar{x}\|^2 \quad \textrm{ holds for all } x\in U.$$
}
\end{defn}

%We record the following change of coordinates formula for reference.
%\begin{thm}[{\cite[Exercise 6.7]{VA}}]\label{thm:change_coordinates}
%Let $U\subset\R^n$ and $\widetilde{U}\subset\R^m$ be open sets and let $F\colon U\rightarrow\widetilde{U}$ be a smooth map. Let $C=F^{-1}(D)$ for a set $D\subset\widetilde{U}$ and suppose that $\nabla F(\bar{x})$ has full rank $m$ 
%at a point $\bar{x}\in C$. Then $$N_C(\bar{x})=\nabla F(\bar{x})^{*}N_D(F(\bar{x})),$$ $$\hat{N}_C(\bar{x})=\nabla F(\bar{x})^{*}\hat{N}_D(F(\bar{x})).$$ 
%\end{thm}

Functional counterparts of normal cones are subdifferentials.
\begin{defn}[Subdifferentials]\label{defn:subdiff}
{\rm Consider a function $f\colon\R^n\rightarrow\overline{\R}$ and a point $\bar{x}\in\R^n$ where $f$ is finite. The {\em regular} and {\em limiting subdifferentials} of $f$ at $\bar{x}$, respectively, are defined by 
$$\hat{\partial}f(\bar{x})= \{v\in\R^n: (v,-1)\in \hat{N}_{\mbox{{\scriptsize {\rm epi}}}\, f}(\bar{x},f(\bar{x}))\},$$  
$$\partial f(\bar{x})= \{v\in\R^n: (v,-1)\in N_{\mbox{{\scriptsize {\rm epi}}}\, f}(\bar{x},f(\bar{x}))\}.$$  
%$$\partial_c f(\bar{x})= \{v\in\R^n: (v,-1)\in N^c_{\mbox{{\scriptsize {\rm epi}}}\, f}(\bar{x},f(\bar{x}))\}.$$
The {\em horizon subdifferential} of $f$ at $\bar{x}$ is the set
$$\partial^{\infty} f(\bar{x})= \{v\in\R^n: (v,0)\in N_{\mbox{{\scriptsize {\rm epi}}}\, f}(\bar{x},f(\bar{x}))\}.$$ 
}  
\end{defn}

For $x$ such that $f(x)$ is not finite, we follow the convention that $\hat{\partial}f(x)=\partial f(x)=\partial^{\infty} f(x)=\emptyset$. The subdifferentials $\hat{\partial} f(\bar{x})$ and $\partial f(\bar{x})$ generalize the classical notion of gradient. In particular, for ${C}^1$-smooth functions $f$ on $\R^n$, these two subdifferentials consist only of the gradient $\nabla f(x)$ for each $x\in\R^n$. For convex $f$, these two subdifferentials coincide with the convex subdifferential. The horizon subdifferential $\partial^{\infty} f(\bar{x})$ plays an entirely different role --- it detects horizontal ``normals'' to the epigraph --- and it plays a decisive role in subdifferential calculus. 
See \cite[Corollary 10.9]{VA} or \cite{lag} for more details.
 
\begin{thm}[Fermat \& sum rules]\label{thm:sum_rule}
Consider an lsc function $f\colon\R^n\to\overline{\R}$ and a closed set $Q\subset\R^n$. If $\bar{x}$ is a local minimizer of $f$ on $Q$ and the qualification condition 
$$\partial^{\infty} f(\bar{x})\cap N_Q(\bar{x})=\{0\} \qquad \textrm{ is valid},$$
then the inclusion  $0\in \partial f(\bar{x})+N_Q(\bar{x})$ holds.
\end{thm}

We will also need the following definition in order to guarantee that the subdifferential $\partial f$ adequately reflects properties of the function $f$ itself.
\begin{defn}[Subdifferential continuity]
{\rm A function $f\colon\R^n\to\overline{\R}$ is {\em subdifferentially continuous at} $\bar{x}$ {\em for} $\bar{v}\in\partial f(\bar{x})$ if for any sequences $x_i\to\bar{x}$ and $v_i\to\bar{v}$, with $v_i\in\partial f(x_i)$, it must be the case that $f(x_i)\to f(\bar{x})$.}
\end{defn}

Subdifferential continuity of a function $f$ at $\bar{x}$ for $\bar{v}$ was introduced in \cite[Definition 1.14]{prox_reg}, and it amounts to requiring the function $(x,v)\mapsto f(x)$, restricted to $\gph \partial f$, to be continuous in the usual sense at the point $(\bar{x},\bar{v})$. For example, any strongly amenable (in particular, lsc and convex) function is subdifferentially continuous \cite[Proposition 2.5]{prox_reg}. 

\subsection{Semi-algebraic geometry}
A {\em semi-algebraic} set $S\subset\R^n$ is a finite union of sets of the form $$\{x\in \R^n: P_1(x)=0,\ldots,P_k(x)=0, Q_1(x)<0,\ldots, Q_l(x)<0\},$$ where $P_1,\ldots,P_k$ and $Q_1,\ldots,Q_l$ are polynomials in $n$ variables. In other words, $S$ is a union of finitely many sets, each defined by finitely many polynomial equalities and inequalities. A function $f\colon\R^n\to\overline{\R}$ is {\em semi-algebraic} if $\mbox{\rm epi}\, f\subset\R^{n+1}$ is a semi-algebraic set. For an extensive discussion on semi-algebraic geometry, see the monographs of Basu-Pollack-Roy \cite{ARAG}, Lou van den Dries \cite{LVDB}, and Shiota \cite{Shiota}. For a quick survey, see the article of van den Dries-Miller \cite{DM} and the surveys of Coste \cite{Coste-semi, Coste-min}. Unless otherwise stated, we follow the notation of \cite{DM} and \cite{Coste-semi}. 

A fundamental fact about semi-algebraic sets is provided by the Tarski-Seidenberg Theorem \cite[Theorem 2.3]{Coste-semi}. It states that the image of any semi-algebraic set $Q\subset\R^n$, under a projection to any linear subspace of $\R^n$, is a semi-algebraic set. From this result, it follows that a great many constructions preserve semi-algebraicity. In particular, for a semi-algebraic function $f\colon\R^n\to\overline{\R}$, it is easy to see that all the subdifferential set-valued mappings are semi-algebraic. See for example \cite[Proposition 3.1]{tame_opt}. 

The following two well-known theorems will be of great use for us \cite[Theorems~4.1, 4.5]{DM}. 
\begin{thm}[Semi-algebraic monotonicity]\label{thm:mon}
Consider a semi-algebraic function $f\colon (a,b)\to\R$. Then there are finitely many points $a=t_0 <t_1 <\ldots < t_k =b$ such that the restriction of $f$ to each interval $(t_i,t_{i+1})$ is $C^2$-smooth and either strictly monotone or constant.
\end{thm}

\begin{thm}[Semi-algebraic selection]\label{thm:sel}
Consider a semi-algebraic set-valued mapping $F\colon\R^n\tto\R^m$. Then there is a single-valued semi-algebraic mapping $f\colon\dom F\to\R^m$ satisfying $f(x)\in F(x)$ for all $x\in \dom F$.
\end{thm}

The proof of the following result appears implicitly in \cite[Proposition 4]{Lewis-Clarke}. We outline an argument for completeness.
\begin{lem}[Semi-algebraic chain rule]\label{lem:sa_chain} {\hfill \\}
Consider an lsc semi-algebraic function $f\colon\R^n\to\overline{\R}$ and a semi-algebraic curve $x\colon [0,T)\to\dom f$. Then there exists $\ep >0$ so that both $x$ and $f\circ x$ are $C^2$-smooth on $(0,\ep)$ and for any $t\in (0,\ep)$ we have
\begin{align*}
v\in\sd f(x(t)) \quad &\Longrightarrow\quad \lan v,\dot x(t)\ran =(f\circ x)'(t), \\
v\in\sd^{\infty} f(x(t)) \quad &\Longrightarrow\quad \lan v,\dot x(t)\ran =0.
\end{align*}
\end{lem}
\begin{proof}
By Theorem~\ref{thm:mon}, there exists $\ep >0$ such that both $x$ and $f\circ x$ are $C^2$-smooth on $(0,\ep)$. Let $\phi:=f\circ x$ and define $$\mathcal{M}:=\{(x(t),\phi(t)):\; t\in (0,\ep)\}.$$ Clearly being a graph of a $C^2$-smooth semi-algebraic function, the set $\mathcal{M}$ is a semi-algebraic $C^2$-submanifold of the epigraph $\epi f$.  Taking if necessary a smaller $\ep$, we can be sure that there
is a Whitney $C^2$-stratification of $\epi f$ such that $\mathcal{M}$ is a stratum; see for example \cite[Theorem 4.8]{DM}. Then for
any real $t\in (0,\ep)$, the inclusion  
$$N_{\epi f}(x(t),\phi(t))\subset N_{\mathcal{M}} (x(t), \phi(t)),$$
holds.
On the other hand, we have the representation
$$
N_{\mathcal{M}} (x(t), \phi(t))=\{(v,\al):\; \lan v,\dot x(t)\ran +\al\phi '(t)=0\}.
$$
Finally recalling Definition~\ref{defn:subdiff},
the result follows immediately. \qed
\end{proof}

\section{Strong subregularity and growth of semi-algebraic functions}
\label{sec:main}
We are now ready to prove the main result of the current work. We should note that the implication $1\Rightarrow 2$ in the theorem below is true without semi-algebraicity (see Theorem~\ref{thm:gen},\cite[Theorem~3.1]{Dima_Ng},\cite[Theorem~2.1]{AG_infin}). The semi-algebraic setting, on the other hand, allows us to provide an appealing geometric argument of this result. The implication $2\Rightarrow 1$, in contrast, may easily fail when the function in question is not semi-algebraic; see Example~\ref{ex:fail}.
\begin{thm}[Strong metric subregularity and quadratic growth]\label{thm:mainres} {\hfill \\ }
Consider an lsc, semi-algebraic function $f\colon\R^n\to\overline{\R}$ and a point $\bar{x}\in\R^n$ that is a local minimizer of $f$. Consider the following two statements:
\begin{enumerate}
\item Subdifferential $\sd f$  is strongly subregular at $(\bar{x},0)$ with modulus $\kappa$,
\item There exist real numbers $\al>0$ and $\ep >0$ such that the inequality
$$f(x)\geq f(\bar{x})+\frac{\alpha}{2} \|x-\bar{x}\|^2\quad \textrm{ holds for all } x\in \B_{\ep}(\bar{x}).$$
\end{enumerate}
Then the implication $1\Rightarrow 2$ holds where we can choose $\alpha$ arbitrarily in $(0,\kappa^{-1})$. The converse implication $2\Rightarrow 1$ holds provided that $f$ is subdifferentially continuous at $\bar{x}$ for $\bar{v}=0$.
\end{thm}

\begin{proof}
Without loss of generality, assume $\bar{x}=0$ and $f(\bar{x})=0$.

\noindent \underline{$1\Rightarrow 2$:} Suppose that the subdifferential $\sd f$  is strongly metrically subregular at $(\bar{x},0)$ with modulus $\kappa$ and 
define the function $$\vf(t):=\inf\{ f(x):\; \| x\|=t\}.$$
Standard arguments using quantifier elimination show that $\vf$ is semi-algebraic (see the discussion in \cite[Section 2.1.2]{Coste-semi}). It follows from Theorem~\ref{thm:mon} that $\vf$ is $C^2$-smooth for all sufficiently small $t$. If $\lim_{t\searrow 0}\vf(t)>0$,  or $\vf(t)\to 0$ but $\lim_{t\searrow 0}\dot\vf(t)$ is positive, then the theorem obviously holds.
So we assume $\vf(t)\to 0$ and $\dot{\vf}(t)\to 0$  as $t\searrow 0$.

Note that the infimum in the definition of $\vf$ is attained if $\vf(t)$ is finite.
Applying Theorem~\ref{thm:sel} to the mapping $t\mapsto\argmin\{ f(x):\; \| x\|=t\}$, there is a semi-algebraic mapping $x(t)$ such that $f(x(t))=\vf(t)$ for all
small $t$.

Applying Lemma~\ref{lem:sa_chain}, we deduce that there is a real $\ep >0$ so that both $x$ and $\vf$ are $C^2$-smooth on $(0,\ep)$ and for any $t\in (0,\ep)$ we have
\begin{align}
v\in\sd f(x(t)) \quad &\Longrightarrow\quad \lan v,\dot x(t)\ran =\vf'(t),\label{eqn:inc} \\ 
v\in\sd^{\infty} f(x(t)) \quad &\Longrightarrow\quad \lan v,\dot x(t)\ran =0.\notag
\end{align}
Observe 
$$ \langle  x(t),\dot x(t)\ran = \frac{1}{2}\frac{d}{dt}\| x(t)\|^2=t,$$
and hence the qualification condition 
$$\partial^{\infty} f(x(t)) \cap N_{\{x:\, \|x\|=t\}}(x(t))=\{0\}$$
holds for all $t\in (0,\ep)$. Consequently, since $x(t)$ minimizes $f$ subject to $\| x\|=t$, applying Theorem~\ref{thm:sum_rule}, we deduce that there is a real number $\la(t)$ satisfying
\begin{equation}\label{eqn:inc2}
\la(t)x(t)\in\sd f(x(t)).
\end{equation}
By strong subregularity, we have  $\|\la(t)x(t)\|\ge \kappa^{-1}\| x(t)\|$, that is $|\la(t)|\ge \kappa^{-1}$.
Finally, combining (\ref{eqn:inc}) and (\ref{eqn:inc2}), we get
$$
 \dot{\vf}(t)  = \la(t) \lan x(t),\dot x(t)\ran =\la(t) \frac{1}{2}\frac{d}{dt}\| x(t)\|^2=\la(t)t.
$$
Since $\bar{x}$ is a local minimizer, we have $\dot\vf(t)\ge 0$. Consequently, we obtain $\la(t)\ge 0$ and $\la(t)\ge \kappa^{-1}$, and hence $\vf(t)\ge \frac{1}{2\kappa}t^2$.

%%%%%%%%%

Suppose now that $f$ is subdifferentially continuous at $\bar{x}$ for $\bar{v}=0$.

\noindent \underline{$2\Rightarrow 1$:} Assume that $2$ holds.
Suppose also for the sake of contradiction that $\partial f$ is not strongly metrically subregular at $(\bar{x},0)$, that is $$\lf_{x\to\bar{x}} \frac{d(0,\partial f(x))}{\|x\|}=0.$$
Define the function 
$$H(t):=\argmin\Big\{\frac{d(0,\partial f(x))}{t}: \|x\|=t\Big\}.$$ 
Then an application of Theorem~\ref{thm:sel} yields a semi-algebraic path $x\colon (0,\ep)\to \R^n$ satisfying $x(t)\in H(t)$ for all $t$. Applying Theorem~\ref{thm:sel} to the mapping $$t\mapsto \argmin\{\|v\|: v\in\partial f(x(t))\},$$
yields a semi-algebraic path $v\colon (0,\ep)\to\R^n$ satisfying $v(t)\in\partial f(x(t))$. 
Observe 
$$\lim_{t\searrow 0} x(t)=0, ~~ \lim_{t\searrow 0} \frac{\|x(t)\|}{\|v(t)\|}=+\infty, ~~ \lim_{t\searrow 0} v(t)=0,$$
where the second equality follows from our assumption that $\partial f$ is not strongly metrically subregular at $(\bar{x},0)$.
Clearly we may extend $x$ and $v$ continuously to $[0,\ep)$.
Decreasing $\ep$, we may assume that on the interval $(0,\ep)$ the composition $f\circ x$ is $C^2$-smooth, $\|v(t)\|$ is non-decreasing, and that $\dot{x}(t)$ is nonzero. Moreover since $f$ is subdifferentially continuous at $\bar{x}$ for $0$, we have $\lim_{t\searrow 0} f(x(t))=0$. Hence the composition $f\circ x$ is continuous on $[0,\ep)$.

We now reparametrize $x$ by arclength. Namely, since $x$ is semi-algebraic and bounded, an application of Theorem~\ref{thm:mon} implies that $x$ has finite length
$$L:=\int^{\ep}_{0} \|\dot{x}(t)\|\; dt.$$
Define the function $s\colon [0,\ep]\to [0,L]$ by setting 
$$s(r):=\int^{r}_{0} \|\dot{x}(t)\|\; dt.$$
Let $y(\tau):= x(s^{-1}(\tau))$ and $\omega(\tau):=v(s^{-1}(\tau))$.
Clearly $y$ is $C^2$-smooth on $(0,L)$ and satisfies $\|\dot{y}(\tau)\|=1$ for all $\tau \in (0,L)$.

We successively conclude, using Lemma~\ref{lem:sa_chain}  for the second equality,
\begin{align*}
\frac{\alpha}{2} \|y(\tau)\|^2\leq f(y(\tau))&=\int_0^{\tau} \frac{d}{dr}(f\circ y)(r) \; dr=\int_0^\tau \langle \dot{y}(r),\omega(r)\rangle \; dr\\
&\leq \int_{0}^{\tau}\|\dot{y}(r)\|\cdot \|\omega(r)\|\; dr = \int_{0}^{\tau} \|\omega(r)\| \; dr\\
&\leq \tau \|\omega(\tau)\|.
\end{align*}
We deduce
\begin{equation}\label{eqn:key}
0< \frac{\alpha}{2} \leq \frac{\tau}{\|y(\tau)\|}\frac{\|\omega(\tau)\|}{\|y(\tau)\|}.
\end{equation}
Now the mean value theorem for vector-valued functions implies 
$$y(\tau)=\Big(\int^1_{0} \dot{y}(h\tau)\; dh\Big)\tau.$$
Hence we have 
$$\frac{\sqrt{n}\cdot \|y(\tau)\|}{\tau} \geq \frac{\|y(\tau)\|_{1}}{\tau}=\sum^n_{i=1} \Big|\int^1_{0} \dot{y}_i(h\tau)\; dh\Big|.$$
Since $x$ is semi-algebraic, so is the derivative $\dot{x}$. Applying Theorem~\ref{thm:mon}, we deduce that there exists $\delta >0$ such that each function $\dot{y}_i$ has a constant sign on $(0,\delta)$. Hence for $\tau\in (0,\delta)$, we obtain 
$$\sum^n_{i=1} \Big|\int^1_{0} \dot{y}_i(h\tau)\; dh\Big|=\int^1_{0} \sum^n_{i=1} |\dot{y}_i(h\tau)| \; dh\geq 1.$$
We conclude that the quantity $\frac{\tau}{\|y(\tau)\|}$ is bounded for small $\tau$. Letting $\tau$ tend to zero in (\ref{eqn:key}), we arrive at a contradiction. \qed
\end{proof}

The following examples show that the implication $2\Rightarrow 1$ of  Theorem~\ref{thm:mainres} may easily fail in absence of subdifferential continuity or semi-algebraicity.
\begin{exa}[Equivalence fails without subdifferential continuity] {\hfill \\}
{\rm Consider the lsc, semi-algebraic function $f\colon\R\to\R$ defined by 
$$f(x)= \left\{
        \begin{array}{ll}
            1+x^4, & \quad x < 0, \\
            x^2, & \quad x \geq 0.
        \end{array}
    \right.$$ 
Observe that $f$ is not subdifferentially continuous at $\bar{x}=0$ for $\bar{v}=0$. Clearly $0$ is a strong local minimizer of $f$. However, one can easily check that $\partial f$ is not strongly metrically subregular at $(0,0)$.
}
\end{exa}

\begin{exa}[Equivalence fails without semi-algebraicity]{\hfill \\} \label{ex:fail}
{\rm
Consider the function $f\colon\R\to\R$ defined by $$f(x):=2x^2+\frac{1}{2}x^2\sin\Big(\frac{1}{x}\Big).$$
It is easy to check that the origin is a strong local minimizer of $f$, while $\partial f$ is not strongly metrically subregular at $(0,0)$. 
}
\end{exa}

\begin{rem}
{\rm
During the review process, the Associate Editor posed the question of whether the analogue of Theorem~\ref{thm:mainres} holds pairing a quadratic growth condition at a not necessarily isolated minimizer and ``subregularity'' of the subdifferential (see Definition~\ref{defn:subreg}). More precisely when is existence of $\alpha >0$ and $\varepsilon >0$ such that 
$$
f(x)\ge f(\xb) +\frac{\alpha}{2} \,d\big(x;(\partial f)^{-1}(0)\big)^2\quad \textrm{ holds for all } x\in\B_{\ep}(\bar{x})
$$
equivalent to the condition
\begin{equation*}
d(x; (\partial f)^{-1}(0))\le\kk d\big(0;\partial f(x)\big)\quad\quad\mbox{ for all}\quad x\in\B_{\delta}(\bar{x}),
\end{equation*}
for some $\kappa >0$ and $\delta >0$.
We will see in the next section that subregularity of the subdifferential at a minimizer always entails a quadratic growth condition. In line with Theorem~\ref{thm:mainres}, we conjecture  that the converse holds for semi-algebraic functions. 
}
\end{rem}

\section{Metric subregularity entails quadratic growth}\label{sec:imply}
In this section, we observe that strong subregularity of a subdifferential, for virtually any subdifferential of interest, at a local minimizer implies strong local minimality. In fact, we will see that an analogues statement holds even for (not necessarily strong) subregularity of the subdifferential (see Definition~\ref{defn:subreg}). In essence, the argument we present is a clarification of the proofs of \cite[Theorem~2.1]{AG_infin} and \cite[Theorem~3.1]{Dima_Ng}, making them shorter and technically simpler.
For simplicity and more inline with the current paper, we work within the finite dimensional setting, and comment on the infinite dimensional  analogue at the very end.   
We begin with metric subregularity, an extension of the strong metric subregularity concept. 
\begin{defn}[Metric subregularity]\label{defn:subreg}
{\rm
A set-valued mapping $F\colon\R^n\rightrightarrows \R^m$ is {\em metrically subregular} at $(\bar{x},\bar{y})\in\gph F$ {\em with modulus } $\kappa \geq 0$ if there is a real number $\varepsilon >0$ such that the inequality
$$d\big(x, F^{-1}(\bar{y})\big)\leq \kappa d\big(\bar{y}; F(x)\big)\qquad \textrm{ holds for all} \quad x\in\B_{\varepsilon}(\bar{x}).$$
}
\end{defn}

In particular, when $F^{-1}(\bar{y})$ is a singleton, metric subregularity reduces to strong metric subregularity. 
We will need the following simple, albeit telling, result asserting that certain small Lipschitz perturbations do not destroy metric subregularity.
\begin{lem}[Metric subregularity under perturbations]\label{lem:pert} {\hfill \\}
Consider a set-valued mapping $F: \R^n\rra \R^m$ that is metrically subregular at $\xyb\in \gph F$ with modulus $\kappa >0$. Specifically 
fix a real number $\ep>0$ such that the inequality 
\begin{equation*}
d(x; F^{-1}(\bar{y}))\le\kk d\big(\oy;F(x)\big)\quad\quad\mbox{ holds for all}\quad x\in\B_{\ep}(\bar{x}).
\end{equation*}
Consider further a mapping $G: \R^n\rra \R^m$ satisfying
$$G(x)\subset \ell d(x; F^{-1}(\bar{y}))\B \quad\textrm{ for all } x\in\B_{\ep}(\bar{x}),$$ for some real number $ 0\leq \ell <\kappa^{-1}.$ Then $F+G$ is metrically subregular at $\xyb$ with modulus $(\kappa^{-1}-\ell)^{-1}$. More specifically, the inequality 
$$
d(x; (F+G)^{-1}(\bar{y}))\le (\kappa^{-1}-\ell)^{-1} d(\yb; (F+G)(x)) \quad \textrm{ holds for all } x\in \B_{\ep}(\xb).
$$
\end{lem}

\begin{proof}
For any  $x\in \B_{\ep}(\xb)$, we successively deduce
\begin{align*}
d(\yb; F(x)+G(x))&\ge d(\yb; F(x))-\ell d(x; F^{-1}(\bar{y}))\\
&\ge \kappa^{-1}d(x; F^{-1}(\bar{y}))-\ell d(x; F^{-1}(\bar{y})).
\end{align*}
On the other hand since $G$ vanishes on $F^{-1}(\bar{y})$, we have the inclusion $F^{-1}(\bar{y})\subset (F+G)^{-1}(\bar{y})$.
Consequently $d(x; (F+G)^{-1}(\bar{y}))\leq d(x; F^{-1}(\bar{y}))$. Combining this with the inequality above, the result follows. \qed
\end{proof}

The following is the main result of the current section. As we have alluded to at the onset of the section, the first part of the argument is an elaboration
of the proof of \cite[Theorem~2.1]{AG_infin}, while the second part is to a large extent a clarification of the proof of \cite[Theorem~3.1]{Dima_Ng}, making it shorter and technically simpler.
\begin{thm}[Subregularity of subdifferentials]\label{thm:gen} {\hfill \\}
Consider an lsc function $f\colon\R^n\to\overline{\R}$, which has a local minimum at $\xb\in\dom f$. Suppose that $\hat{\sd} f$ is metrically subregular at $(\xb,0)$ with modulus $\kappa >0$. Then for any real number $\alpha\in (0,\frac{1}{2\kappa})$ there exists $\ep >0$ such that the inequality 
$$
f(x)\ge f(\xb) +\frac{\alpha}{2} \,d\big(x;(\hat{\partial} f)^{-1}(0)\big)^2\quad \textrm{ holds for all } x\in\B_{\ep}(\bar{x}).
$$
If, on the other hand, the subdifferential $\partial f$ is metrically subregular at $(\xb,0)$ with modulus $\kappa >0$, then the inequality above with $\partial f$ replacing $\hat{\partial} f$
holds for any real number $\alpha\in (0,\frac{1}{\kappa})$.
\end{thm}

\begin{proof}
We assume for simplicity  that $\xb=0$ and $f(\xb)=0$. Define for notational convenience $S=(\hat{\partial} f)^{-1}(0)$.%At the first step of the proof we shall show that the conclusion holds with
%$\al<(1/2)[{\rm subreg}\sd f(\xb,0)]^{-1}$. 
%The proof  of this statement consists in
We must estimate the lower bound of $\al>0$ with the property that there is a
sequence $u_k$  satisfying
\begin{equation}\label{eqn:contr}
u_k\to 0,\quad u_k\notin S,\quad   f(u_k)\le\frac{\al}{2} d(u_k;S)^2.
\end{equation}
So fix $\alpha\in (0,\frac{1}{2\kappa})$ and suppose there is such a sequence. Define $\la\in (0,\frac{1}{2})$ by the formula $\frac{\al}{2} = \frac{\la^2}{{\kappa}}$. By Ekeland's principle
for any $k$ there exists $w_k$ such that $\| w_k-u_k\|\le\la  d(u_k; S)$,
$f(w_k)\le f(u_k)$ and $g_k(x)=f(x)+\frac{\la}{{\kappa}} d(u_k; S)\| x- w_k\|$ attains a minimum
at $w_k$.
We see that $g_k$ is a sum of an lsc and a convex continuous function. Consequently, we can find for any $k$ and any $\ep_k> 0$ a point $x_k$
such that
$|f(x_k)-f(w_k)|<\frac{\ep_k}{{\kappa}}d(u_k;S)$, $\| x_k-w_k\|<\frac{\ep_k}{{\kappa}}d(u_k;S)$
and $0\in\hat{\sd} f(x_k) + \frac{\la+\ep_k}{{\kappa}}d(u_k; S)\B$. We deduce $d(0;\hat{\sd} f(x_k))\le \frac{\la+\ep_k}{{\kappa}} d(u_k; S)$.
Observe $d(x_k;S)\ge d(u_k;S)-\| x_k-u_k\|\ge (1-(\la+\frac{\ep_k}{{\kappa}}))d(u_k;S)$, that is
$$
d(0;\hat{\sd} f(x_k))\le \frac{\la+\ep_k}{{\kappa}-({\kappa}\la+\ep_k)}d(x_k;S).
$$
Letting $\ep_k\to 0$ and taking into account the inequality $\lambda <\frac{1}{2}$, 
we arrive at a contradiction to the subregularity of $\hat{\sd} f$.

%Suppose now that exact calculus holds for the subdifferential $\partial$. 

Suppose now that $\partial f$ is subregular at $(\bar{x},0)$. Then the same argument, as above, shows that the theorem holds for any $\al<\frac{1}{2\kappa}$. 
Fix now a real number $\delta\in (0,1)$ and define $f_1(x) := f(x) -\frac{1-\delta}{4\kappa}d(x;S)^2$. By what we have already proved, zero is a local minimizer of $f_1$. Observe $\sd f_1(x)\subset \sd f(x)-\frac{1-\delta}{2\kappa} d(x;S)\sd d(\cdot; S)(x)$. By Lemma~\ref{lem:pert}, the subdifferential
 $\sd f_1$ is subregular at $(0,0)$ with modulus
$\frac{2\kappa}{1+\delta}$. It follows according to what has just been proved that for $\al =\frac{1+\delta-\gamma}{4\kappa}$, with $\gamma>0$ arbitrarily small, there exists $\tau_1>0$ satisfying
 $f_1(x)\ge \frac{1+\delta-\gamma}{8\kappa}d(x;S)^2$ for all $ x \in \tau_1\B$. Hence for such $x$, we have 
$f(x)\geq \Big(\frac{1-\delta}{4}+\frac{1+\delta-\gamma}{8}\Big)\frac{1}{\kappa}d(x;S)^2.$
Repeating this procedure inductively and taking $\delta$ and $\gamma$ arbitrarily close to zero, we obtain the result. \qed
\end{proof}

It is worth noting that the proof of the previous theorem easily extends to infinite dimensional spaces and to virtually any reasonable subdifferential on those spaces. As an aside, we now briefly elaborate on this further. The reader may safely skip this discussion. We begin with some notation. We will let $\X$ denote a Banach space with norm $\|\cdot\|$, while the dual space of $\X$ will be written as $\X^{*}$. The closed unit ball in $\X$ will be denoted by $\B$, while the closed unit ball of $\X^*$ will be written as $\B_{\X^*}$.

There are several types of subdifferentials used in variational
analysis in Banach spaces (proximal, Fr\'{e}chet, limiting Fr\'{e}chet, Dini-Hadamard. G-subdifferential,
generalized gradient). We do not need a general definition here: an interested reader may look into \cite{th_sub_ioffe}. The important point is that most of them can be effectively used 
only in certain classes of Banach spaces. This observation is formally represented by the following
definition playing an important  role in the general theory. 
\begin{defn}[Trustworthy spaces]
{\rm A subdifferential $\sd$ {\it can be  trusted} on a Banach space $\mathcal{X}$ if for any lsc function $f\colon \mathcal{X}\to\overline{\R}$,
any point $\xb\in\dom f$ and any function $g\colon \mathcal{X}\to\overline{\R}$ which is convex and continuous near $\xb$ the following holds:
{\it if $f+g$ attains a local minimum at $\xb$, then for any $\ep>0$ there are $x,u,x^*,u^*$
such that both $x$ and $u$ are $\ep$-close to $\xb$, $f(x)$ is $\ep$-close to $f(\xb)$,
$x^*\in\sd f(x),\; u^*\in\sd g(u)$
and $\| x^*+u^*\|<\ep$.}}
\end{defn}

The Fr\'echet and the limiting Fr\'echet subdifferential can be trusted on Asplund spaces and only on them;
the Dini-Hadamard subdifferential can be trusted on G\^ateaux smooth spaces; the G-subdifferential
and the generalized gradient of Clarke can be trusted on all Banach spaces. See \cite{th_sub_ioffe} for details.

Let us finally agree to say that the {\it exact calculus} holds for a given subdifferential $\sd$ in a 
given Banach space $\mathcal{X}$ if the inclusion
$$
\sd(f+g)(x)\subset \sd f(x)+\sd g(x)
$$
holds whenever $f\colon\mathcal{X}\to\overline{\R}$ is lsc and $g\colon\mathcal{X}\to\overline{\R}$ is Lipschitz continuous near $x$. This is the case when $\partial$ is a ``robust" subdifferential: limiting Fr\'{e}chet subdifferential in an Asplund space, G-subdifferential or generalized gradient on any Banach space.
An appropriate restatement of Theorem~\ref{thm:gen} (with identical proof) in the infinite dimensional setting follows.

\begin{thm}[Subregularity of subdifferentials in Banach spaces] {\hfill \\}
Consider an lsc function $f\colon\mathcal{X}\to\overline{\R}$, defined on a Banach space $\mathcal{X}$, which has a local minimum at $\xb\in\dom f$. Let $\sd$ be a subdifferential trusted on $X$ and suppose that $\sd f$ is metrically subregular at $(\xb,0)$ with modulus $\kappa >0$. Then for any real number $\alpha\in (0,\frac{1}{2\kappa})$ there exists $\ep >0$ such that the inequality 
$$
f(x)\ge f(\xb) +\frac{\alpha}{2} \,d\big(x;(\partial f)^{-1}(0)\big)^2\quad \textrm{ holds for all } x\in\B_{\ep}(\bar{x}).
$$
Moreover, if the exact calculus holds, then the conclusion is valid for any
$\alpha\in (0,\frac{1}{\kappa})$.
\end{thm}

\section{Necessary and sufficient conditions for optimality}\label{sec:nec_suf}
To motivate the discussion, consider a $C^2$-smooth function $f\colon\R^n\to\R$ and a critical point $\bar{x}$ of $f$. Classically then for $\bar{x}$ to be a local minimizer of $f$, the Hessian $\nabla^2 f(\bar{x})$ must be positive semi-definite. On the other hand, if $\nabla^2 f(\bar{x})$ is positive definite, then $\bar{x}$ is guaranteed to be a local minimizer. %It is then tempting to formulate a nonsmooth analogue of these conditions using a second-order subdifferential construction. 
Seeking to extend such a characterization to the nonsmooth setting, it is 
natural to use a second-order subdifferential construction.
However, this seems to invariably lead to characterizations of ``strong metric regularity'' of the subdifferential \cite[Theorem 1.3]{tilt}, \cite{DL_tilt}. To illustrate the limitations of this approach, consider the convex function of two variables $f(x,y)=(|x|+|y|)^2$. Here, the origin is a strong local minimizer, and therefore  Theorem~\ref{thm:mainres} implies that the subdifferential $\partial f$ is strongly subregular at $(\bar{x},0)$. On the other hand, an easy computation shows that $\partial f$ is not even metrically regular at $\bar{x}$, and therefore second-order conditions involving the coderivative of the subdifferential are not applicable. 
% since the origin is a strong local minimizer but not uniformly so (with respect to small linear perturbations). 
In line with the much weaker subregularity concept --- the focus of the current manuscript --- we take a different approach, one that would in particular be applicable to the aforementioned function. To motivate our main result, we first revisit the standard second-order optimality conditions and restate them in terms of the gradient itself, rather than the Hessian. In what follows, we let the symbol $I$ denote the identity mapping.
\begin{prop}[Standard second-order conditions revisited]{\hfill \\ }
Consider a $C^2$-smooth function $f$ on $\R^n$ and let $\bar{x}$ be a critical point. Then the following are true.
\begin{enumerate}
\item  The matrix $\nabla^2 f(\bar{x})$ is positive semi-definite if and only if the mapping
 \begin{equation}\label{eq:nec}
\nabla f+r(I-\bar{x}) \quad\textrm{ is strongly subregular at } (\bar{x},0) \textrm{ for all } r>0.
\end{equation}
\item The matrix $\nabla^2 f(\bar{x})$ is positive definite if and only if there exist constants $\varepsilon >0$, $\lambda >0$, and $r\geq 0$ such that the inequality 
\begin{equation}\label{eq:suff}
\|\nabla f(x)+r(x-\bar{x})\|\geq (\lambda +r)\|x-\bar{x}\|\quad \textrm{ holds for all }  x\in\B_{\varepsilon}(\bar{x}).
\end{equation}
\end{enumerate}
\end{prop}
\begin{proof}
Clearly the Hessian $\nabla^2 f(\bar{x})$ is positive semi-definite, if and only if $\nabla^2 f(\bar{x})+rI$ is nonsingular for all $r >0$.
Since the matrix $\nabla^2 f(\bar{x})+rI$ is the Jacobian of the mapping $x\mapsto \nabla f(x)+r(x-\bar{x})$, the first claim follows.

To see the second claim, observe
$$\nabla f(x)=\nabla^2 f(\bar{x})(x-\bar{x})+o(\|x-\bar{x}\|),$$
%  $\nabla^2 f(\bar{x})$ is positive definite.
and hence 
$$\|\nabla f(x)+r(x-\bar{x})\|=\|(\nabla^2 f(\bar{x})+rI) (x-\bar{x})\|+o(\|x-\bar{x}\|)$$
for all real $r$. Consequently if $\nabla^2 f(\bar{x})$ is positive definite, then there exists $\varepsilon >0$ so that $\|\nabla f(x)+r(x-\bar{x})\|\geq (\lambda +r)\|x-\bar{x}\|$ for all $x\in\B_{\varepsilon}(\bar{x})$, where $\lambda >0$ is slightly smaller than the minimal eigenvalue of $\nabla^2 f(\bar{x})$.
To see the converse, suppose \eqref{eq:suff} holds. Then if the Hessian $\nabla^2 f(\bar{x})$ is not positive semi-definite, we may choose an eigenvalue $\lambda' <0$ and a corresponding  eigenvector $v\in\R^n$ with $\|v\|=1$. Then in terms of $z_t=\bar{x}+tv$, we have
\begin{align*}
(\lambda +r)t&\leq \|\nabla f(z_t)+r(z_t-\bar{x})\|=\|(\nabla^2 f(\bar{x})+rI) (z_t-\bar{x})\|+o(\|z_t-\bar{x}\|)\\
&\leq (\lambda'+r) t+ o(t)
\end{align*}
a contradiction.
\qed 
\end{proof}

%Observe
%$$\|\nabla f(x)+r(x-\bar{x})\|=\|(\nabla^2 f(\bar{x})+rI)(x-\bar{x})\| +o(\|x-\bar{x}\|).$$ An easy computation now shows that positive definiteness of $\nabla^2 f(\bar{x})$ asserts existence of constants $\epsilon >0$, $\lambda >0$, and $r\geq 0$ such that the inequality 
%\begin{equation}\label{eq:suff}
%\|\nabla f(x)+r(x-\bar{x})\|\geq (\lambda +r)\|x-\bar{x}\|\quad \textrm{ holds for all }  x\in\B_{\epsilon}(\bar{x}).
%\end{equation}
%The latter is simply a strong subregularity condition with a certain dependence of the modulus on the perturbation parameter $r$.

The necessary condition \eqref{eq:nec} and the sufficient condition \eqref{eq:suff} for optimality of $C^2$-functions are rather striking since they do not invoke any second-order construction. We will see shortly that these conditions, simply restated in subdifferential terms, ``almost'' serve the same purpose for nonsmooth functions. This generalization is surprisingly subtle, however. To illustrate, consider the univariate $C^1$-smooth function $f(x)=x^{\frac{3}{2}}$. Clearly, the origin is a critical point of $f$, and moreover it is easy to see that condition  \eqref{eq:suff} holds. On the other hand, the origin is not a local minimizer of $f$. The difficulty, absent in the $C^2$-smooth setting, is due to nonexistence of parabolic  minorants. Hence an extra condition, in addition to  \eqref{eq:nec} and \eqref{eq:suff},  addressing this situation will be required. The semi-algebraic setting once again provides a nice motivation for the development. To this end, recall that the {\em gauge} of any closed set $K\subset \R^n$ is the function 
$$\Gamma_K(x)=\inf\{\lambda >0: x\in\lambda K\}.$$
Whenever $\Gamma_K(x)$ is finite, the infimum in the definition is attained.
For any closed-valued set-valued mapping $K\colon\R^n\rightrightarrows\R^m$, we abuse notation slightly and define the {\em parametric gauge function} to be
$$\Gamma_K(x;y):=\Gamma_{K(x)}(y),$$
or more precisely 
$\Gamma_K(x;y)=\inf\{\lambda >0: y\in\lambda K(x)\}$.
%We note that a second order sufficient condition for optimality (based on coderivative constructions under a prox-regularity requirement) appears in \cite[Theorem 1.3]{tilt}. That work, on the other hand, was motivated more by sensitivity analysis and in particular ``strong metric regularity'' of the subdifferential. 
%Since our work revolves around the much weaker strong subregularity concept, the sufficient condition we obtain is much closer to being necessary.
The gauge of the subdifferential arises naturally when considering minorants of semi-algebraic functions.

\begin{lem}[Minorants of semi-algebraic functions]\label{lem:lower} \hfill \\
Consider semi-algebraic functions $f\colon\R^n\to\overline{\R}$ and
$\psi\colon\R_+\to\R$, satisfying $$f(\bar{x})=\psi(0)\qquad\textrm{ and }\qquad f(x) \geq \psi(\|x-\bar{x}\|)~~\textrm{ for all }x \textrm{ near } \bar{x}.$$ Then the inequality
$$\frac{\psi'(\|x-\bar{x}\|)}{\|x-\bar{x}\|} \leq -\frac{1}{\Gamma_{\partial f}(x;\bar{x}-x)}$$
holds for all $x \textrm{ near }\bar{x}$ for which $\Gamma_{\partial f}(x;\bar{x}-x)$ is finite.
\end{lem}
\begin{proof}
Define the semi-algebraic set-valued mapping $H\colon\R_+\rightrightarrows\R^n\times\R_+$ by
$$H(t)=\{(x,\lambda): \Gamma_{\partial f}(x;\bar{x}-x)= \lambda \quad\textrm{and}\quad \|x-\bar{x}\|=t\}.$$	
Suppose for the sake of contradiction that the lemma is false. Then an easy argument using Theorem~\ref{thm:sel} immediately implies that there exists a semi-algebraic selection $(x(t),\lambda(t))$ of $H$, defined at least on some open interval $(0,\eta)$, satisfying 
$\frac{\psi'(t)}{t} > -\frac{1}{\lambda(t)}$.
Observe that for all sufficiently small $t >0$, the inequality 
$$(f\circ x)(t)\geq \psi(t) \qquad\textrm{ holds}.$$
Consequently, appealing to Theorem~\ref{thm:mon}, we deduce 
$$(f\circ x)'(t)\geq \psi'(t) \qquad \textrm{ for all small } t>0.$$
Combining this inequality with Lemma~\ref{lem:sa_chain}, we obtain 
\begin{align*}
\psi'(t)&\leq \langle \partial f(x(t)),\dot{x}(t)\rangle=-\lambda(t)^{-1}\langle x(t)-\bar{x},\dot{x}(t)\rangle\\ &=-\frac{1}{2\lambda(t)}\frac{d}{dt}\|x(t)-\bar{x}\|^2=-\frac{1}{\lambda(t)}t.
\end{align*}
This is a clear contradiction and the result follows. \qed
\end{proof}

Specializing the previous lemma to parabolic minorants yields the following important observation.

\begin{cor}[Parabolic minorants of semi-algebraic functions]\label{cor:quad} \hfill \\
Consider a semi-algebraic function $f\colon\R^n\to\overline{\R}$ that is finite at a point $\bar{x}$, and let $r>0$ be a real number. Suppose that for some $\varepsilon >0$ we have $$f(x) \geq f(\bar{x})-\frac{r}{2}\|x-\bar{x}\|^2 \qquad\textrm{ for all } x\in\B_{\varepsilon}(\bar{x}).$$ Then the inequality
$$\Gamma_{\partial f}(x;\bar{x}-x) \geq \frac{1}{r}\qquad
\textrm{ holds for all } x \textrm{ near }\bar{x}.$$	
\end{cor}
\begin{proof}	
Apply Lemma~\ref{lem:lower} to the function $\psi(t)=f(\bar{x})-\frac{r}{2}t^2$. \qed	
\end{proof}
	
We now arrive to the main theorem of this section. This result provides necessary conditions and sufficient conditions for minimality of nonsmooth functions purely in subdifferential terms.
As in the $C^2$-smooth case, there is a small gap between these two types of conditions. 

\begin{thm}[Necessary and sufficient conditions for optimality]\label{thm:nec_suff}{\hfill \\ }
Consider an lsc function $f\colon\R^n\to\overline{\R}$ that is finite at a point $\bar{x}$ satisfying $0\in\hat{\partial} f(\bar{x})$. %Define also the parametric gauge functional 
%$$\Gamma_{\hat{\partial} f}(x;y)=\inf\{\rho >0: y\in \rho \cdot \hat{\partial} f(x)\}.$$
Then the following are true:
\begin{enumerate}
\item\label{it:a} If $f$ is in addition semi-algebraic and subdifferentially continuous at $\bar{x}$ for $0$, then necessary conditions for $f$ to have a local minimum at $\bar{x}$ are that for every $r >0$, the following are true:
\begin{enumerate}
\item\label{it:strongsub} The set-valued mapping $\partial f+r(I-\bar{x})$ is strongly subregular at $(\bar{x},0)$.
\item There exists $\varepsilon >0$ such that 
$$\Gamma_{\partial f}(x;\bar{x}-x)\geq \frac{1}{r}\qquad\textrm{ for all } x\in\B_{\varepsilon}(\bar{x}).$$
\end{enumerate}
\item\label{it:b} If there are constants $\varepsilon >0$, $\lambda >0$, and $r \geq 0$ such that 
$$d(0; \partial f(x)+r(x-\bar{x}))\geq (\lambda+r)\|x-\bar{x}\|\quad \textrm{ and }\quad \Gamma_{\hat{\partial} f}(x;\bar{x}-x)\geq \frac{1}{2r}$$
\textrm{ for all } $x\in\B_{\varepsilon}(\bar{x})$,
then  $\bar{x}$ is a strong local minimizer of $f$.
\end{enumerate}
\end{thm}
\begin{proof}
Claim \ref{it:a} is immediate from Theorem~\ref{thm:mainres} and Corollary~\ref{cor:quad}. We now prove claim \ref{it:b}. We assume without loss of generality $\bar{x} = 0$ and $f(\bar{x}) = 0$. Observe 
$$d(0; \partial f(x))+r\|x\|\geq d(0; \partial f(x)+rx)\geq (\lambda+r)\|x\|,$$
and hence $\partial f$ is strongly metrically subregular at $(\bar{x},0)$.
Hence by Theorem 4.3, it is sufficient now to show that $f$ attains a local minimum at $\bar{x}$. Consider
the function $\rho(t) =\inf\{f(x): \|x\|\leq t\}$. Clearly, $\rho$ is lsc and non-increasing and
$\rho(0) = 0$. If $\rho(t) = 0$ for some positive $t$, we are done. So we assume by way of
contradiction that $\rho(t) < 0$ for all small $t >0$.

Note $\rho(t) = o(t)$ since we have $0\in\hat{\partial} f(\bar{x})$. 
As $\rho$ is non-increasing, it is almost
everywhere differentiable. 
Consequently $\ls_{t\to 0} \rho'(t)=0$. 
In particular, we may choose $\tau \in (0,\varepsilon)$ and  
$\alpha\in \hat{\partial} \rho(\tau)$ with $-r < \alpha \leq 0$. Let $x$ be such that $f(x) = \rho(\tau)$ and $\|x\|\leq \tau$.
In the case $\tau':=\|x\|<\tau$, the function $\rho$ is constant on $[\tau',\tau]$ and therefore $0\in\hat{\partial} f(\tau')$. We may then replace $\tau$ with $\tau'$ and $\alpha$ with zero to in any case ensure $\|x\|= \tau$.
Finally define the unit vector $u:=\frac{1}{\tau}x$.

Let $Y$ be the subspace orthogonal to $u$ and consider the following diffeomorphism of a neighborhood of $(\tau,0)\in\R\times Y$ onto a neighborhood of $x$: 
$$\Psi(\zeta,y)=\sqrt{\zeta^2-\|y\|^2}\,u+y.$$
Observe $\Psi(\tau,0)=x$. Moreover we have $\|\Psi(\zeta,y)\|=t$ if and only if $\zeta=t$.
Set now $$g:=f\circ\Psi\quad\textrm{ and }\quad \eta(t):=\min_y g(t,y).$$
Observe $\eta(\tau)=\rho(\tau)$ and $\eta(t)\geq \rho(t)$.
We deduce the inclusion $\alpha\in \hat{\partial} \eta(\tau)$.
By \cite[Proposition~3.1]{IP}, this implies $(\alpha,0)\in\hat{\partial} g(\tau,0)$. Standard calculus, in turn, yields
$$\hat{\partial} g(\tau,0)\subset (\nabla \Psi (\tau,0))^* \hat{\partial} f(x).$$
It follows that there is a vector $x^*\in\hat{\partial} f(x)$ satisfying $\langle x^*,u\rangle=\alpha \leq 0$ and whose projection onto $Y$ is zero. We deduce $x^*=\alpha u$ and consequently 
$$d(0;\partial f(x)+rx)\leq \|x^*+rx\|=\Big|\frac{\alpha}{\|x\|}+r\Big|\cdot\|x\|.$$
Note now the inclusion $\frac{\alpha}{\|x\|}x\in \hat{\partial} f(x)$. Appealing  to the  assumed inequality $\Gamma_{\hat{\partial} f}(x;-x)\geq \frac{1}{2r}$, we deduce $\Big|\frac{\alpha}{\|x\|}+r\Big|\leq r$.
Thus we have arrived at a contradiction $d(0;\partial f(x)+rx)\leq r\|x\|$, which completes the proof. \qed
\end{proof}

It is worth emphasizing that unlike the standard second order conditions, the previous theorem does not even require the existence of second derivatives.

%\begin{corollary}
%Consider an lsc function $f\colon\R^n\to\overline{\R}$ that %is finite at $\bar{x}$. If the inclusion $0\in\partial f(\bar{x})$ holds and there exist $\varepsilon >0$, $\lambda >0$, and $r >0$ such that 
%$$d(0; \partial f(x)+r(x-\bar{x}))\geq (\lambda+r)\|%x-\bar{x}\|  \qquad \textrm{ for all } x\in\B_{\varepsilon}(\bar{x}),$$
%then $\bar{x}$ is a strong local minimizer of $f$.
%\end{corollary}
%\begin{proof}
%The assumption $0\in\hat{\partial} f(\bar{x})$ in Theorem~\ref{thm:nec_suff} was only needed to show $\rho(t)=o(t)$. Hence if this is true, then we are done. Otherwise we have
%$$\lf_{x\to\bar{x}} \frac{f(x)-f(\bar{x})}{\|x-\bar{x}\|^2} =-\infty.$$
%there exists a sufficiently small $\alpha >0$ so that the function
%$$g(x)=f(x)+\frac{r}{2}\|x-\bar{x}\|^2$$
%when minimized on the ball $\B_{\alpha}(\bar{x})$ achieves  minimum at a point 
%\end{proof}

\begin{acknowledgements}
We thank the anonymous referees for their careful reading and insightful comments, which have improved the quality of the paper. 
Much of the current manuscript was written while the first author was visiting the ECE department at Northwestern University. The first author thanks the department, and Jorge Nocedal in particular, for the hospitality. 
\end{acknowledgements}

% BibTeX users please use one of
%\bibliographystyle{spbasic}      % basic style, author-year citations
%\bibliographystyle{spmpsci}      % mathematics and physical sciences
%\bibliographystyle{spphys}       % APS-like style for physics
%\bibliography{}   % name your BibTeX data base

% Non-BibTeX users please use

\bibliographystyle{plain}
\small
\parsep 0pt
\bibliography{bibliography}

\end{document}